\documentclass[a4paper]{article}

\usepackage[left=2.2cm, right=2.2cm, top=2.5cm, bottom=2.5cm]{geometry}

\usepackage{amsfonts, amsmath, hyperref, latexsym, amssymb, amsthm, url,tikz, xcolor, bm, ifthen, dsfont}
\usepackage{setspace}
\usepackage{caption}
\numberwithin{equation}{section}

\usetikzlibrary{shapes.geometric}
\usetikzlibrary{positioning}
\newtheorem{theorem}{Theorem}[section]
\newtheorem{lemma}[theorem]{Lemma}

\newtheorem{proposition} [theorem]{Proposition}

\newtheorem{remark}[theorem]{Remark}
\newcommand{\N}{\mathcal{N}}
\renewcommand{\P}{\mathbb{P}}
\newcommand{\E}{\mathbb{E}}
\def\1{\mathds{1}}
\renewcommand{\Re}{\mathrm{Re}\,}
\newcommand{\U}[1]{\mathrm{(U}_{#1}\mathrm{)}}
\newcommand{\R}{\mathbb{R}}
\newcommand{\Ext}{\mathcal{E}}

\title{Moments of P\'olya urns balanced in expectation}
\author{Colin Desmarais\\
{\small TU Wien}\\[-0.8ex] 
{\small Institute of Discrete Mathematics and Geometry}\\[-0.8ex]
{\small Wiedner Hauptstrasse 8-10, 1040 Vienna, Austria}\\ 
{\small \tt colin.desmarais@tuwien.ac.at}}

\begin{document}

\maketitle

\begin{abstract}
In this work, recent results on the moments of balanced P\'olya urns are generalized to unbalanced urns, with the condition that the expected change in total activity at each step is constant. We also provide applications of our results to the degree distributions of random trees grown by uniform attachment with freezing and to the degree distribution of hooking networks; in both cases we prove a normal limit law with convergence of all moments.
\end{abstract}

\section{Introduction}

A generalized P\'olya urn contains several balls of different colours. At each step, a ball is selected from the urn at random proportional to its {\em activity}, a real value assigned to each ball of the same colour. The ball is then replaced with a (potentially random) set of other balls which depends on the colour of the drawn ball. The expected replacement scheme at each step is encoded in an {\em intensity matrix} $A$, defined below. As is well-studied, the eigenvalues of this matrix influence the asymptotic behaviour of the urn. Under some irreducibility assumptions, the intensity matrix $A$ admits a simple real eigenvalue $\lambda_1$ such that $\lambda_1 \geq \Re \lambda$ for all other eigenvalues $\lambda$. If $\lambda_1 \geq 2 \Re \lambda$, the urn is called {\em small} and in several cases is asymptotically normal, see \cite[Theorems 3.22 \& 3.23]{JANS:04} for the general result. It was further shown in several recent works that convergence of moments also holds in these cases if the urn is {\em balanced}; that is, if the change in the sum of the activities of the balls at each step is constant \cite{JANS:20, JANS:25, JAPO:18}; we note also the work of Pouyanne \cite{POUY:08} on the moments of balanced {\em large urns} (where $\lambda_1 < 2 \Re \lambda$ for some eigenvalue $\lambda \neq \lambda_1$ of $A$).

\medskip

Despite several works on the moments of balanced urns, similar results for unbalanced urns have remained elusive (see \cite[Problem 1.1]{JANS:20} and \cite[Problem 1.1]{JANS:25}). In this work we generalize several results to a class of unbalanced urns. Our results are stated in slightly more generality (see Section \ref{sec:results}), but the main result can be summarized as showing convergence of moments for strictly small urns (where $\lambda_1 > 2 \Re \lambda$ for all other eigenvalues of $A$) if we assume the change in the sum of the activities of the balls at each step is constant in expectation, what we call {\em balanced in expectation} (see \eqref{eq:iidBalancedDef} below). 

\medskip

To prove our results, we perform a similar decomposition of the P\'olya urn appearing in previous works, notably in \cite{BAHU:99, BAHU:05, JANS:20, JANS:25}. The idea is to rewrite the contribution of the $k$-th draw on the composition of the urn at a later time as a linear transformation of a martingale difference sequence. In the balanced case, the linear transformation is deterministic (not random), and so amenable to analysis; this property does not hold in greater generality. However, when the urn is balanced in expectation, the contribution at the $k$-th draw can be further decomposed into a part that maintains the martingale difference sequence property, and a noisy part which can be controlled, all while maintaining the deterministic linear transformation applied to the contributions (see \eqref{eq:XDecomp}).

\medskip

P\'olya urns have long been applied to classes of random increasing trees and their generalizations, and used to prove normal limit laws for their degrees; see for example \cite{HOJS:17, JANS:05, MAHM:19, MASM:92, MASS:93}. In all of the examples cited, the urns used are balanced and satisfy the conditions of \cite{JANS:20, JANS:25, JAPO:18} to guarantee the convergence of the moments as well. However, some recently introduced models of increasing trees and generalizations thereof do not grow in a `balanced-like' manner, but do grow in a way that is `balanced in expectation'; these include for example trees constructed via uniform attachment with freezing, and hooking networks. In Section \ref{sec:applications}, we prove central limit theorems and convergences of moments for the degree distributions of these models. 

\medskip

We gather relevant terminology and assumptions in Section \ref{sec:definitions}, mostly copying the summaries provided in \cite{JANS:20, JANS:25} (though with the assumption of balanced urns replaced with urns balanced in expectation). Our main results are presented in Section \ref{sec:results}; these include a general bound for moments of urns balanced in expectation, convergence of moments for strictly small urns admitting a central limit theorem, and asymptotics for the expectations of urns balanced in expectation. The main technical work of the proofs is contained in Section \ref{sec:momentBounds}, and the completions of the proofs are contained in Section \ref{sec:proofs}. Applications of our results are provided in Section \ref{sec:applications}. 

\bigskip

\noindent{\bf Acknowledgements.} The author would like to thank Svante Janson for helpful discussions in preparation of this article. 
\medskip

\noindent{\bf Funding statement.} This research was funded in whole or in part by the Austrian Science
Fund (FWF) [10.55776/PAT6732623]. For open access purposes, the author
has applied a CC BY public copyright licence to any author accepted
manuscript version arising from this submission.

\section{Definitions, notation and assumptions}\label{sec:definitions}

\noindent{\bf Some general notation.} Throughout this work, all vectors are column vectors, with $v^T$ used to denote the (row vector) transpose of $v$. We will use $\cdot$ to denote the Euclidean inner product of two vectors. 

The standard Euclidean norm for vectors is denoted $| \cdot |$, while we use $||M||$ to denote the operator norm of the matrix $M$. For a random variable or random vector $X$, the $L^p$-norm of $X$ is denoted $||X||_p := \left(\E|X|^p\right)^\frac{1}{p}$. 

We reserve the letter $q$ for the dimension of our vectors and matrices, i.e., all vectors belong to $\mathbb{F}^q$ for some field $q$, and all matrices $A$ belong to $\mathbb{F}^{q \times q}$. The letter $p$ is reserved for norms; i.e., when denoting the $L^p$-norm $|| \cdot ||_p$.

\bigskip

\noindent {\bf Generalized P\'olya urns.} We suppose there are $q \geq 2$ types (or colours) of balls, each with a non-negative activity $a_i \in \R_{\geq 0}$ and assigned a random vector $\xi_i = (\xi_{i,1}, \ldots, \xi_{i,q})^T$. The urn at time $n \geq 0$ is denoted by the non-negative vector 
\[X_n = (X_{n,1}, \ldots, X_{n,q})^T \in \R_{\geq 0}^q.\]
Set the {\em activity vector} of the urn as $a := (a_1, \ldots, a_q)^T$, and define the {\em total activity of the urn} at time $n$ to be the value
\begin{equation}\label{eq:SDef}
S_n := a \cdot X_n = \sum_{i=1}^q a_i X_{n,i}.
\end{equation}
The composition of the urn $(X_n)_{n=0}^\infty$ evolves as a discrete-time Markov process as follows. We initiate with a non-zero and non-random vector $X_0$ such that 
\[ S_0 = a \cdot X_0 > 0.\]
At each step $n \geq 0$, if $S_n > 0$ then a ball is selected from the urn where the probability of selecting a ball of type $i$ for $i=1,\ldots q$ is given by
\begin{equation}\label{eq:BallProbab}
\frac{a_iX_{n,i}}{\sum_{j=1}^q a_j X_{n,j}} = \frac{a_i X_{n,i}}{S_n},
\end{equation}
while if $S_n = 0$, then no ball is selected. We then sample the {\em replacement vector} $\Delta X_n = (\Delta X_{n,1}, \ldots, \Delta X_{n,q})$ where $\Delta X_n \sim \xi_i$ if a ball of type $i$ is selected, and $\Delta X_n = 0$ if no ball is selected. Importantly, the distribution of $\Delta X_n$ depends only on the type of the ball selected at time $n$ (if a ball is selected), independently of anything that has occurred so far. We then update the urn by setting 
\begin{equation}\label{eq:update}
X_{n+1} = X_n + \Delta X_n.
\end{equation}
Note that if the type $i$ has activity $a_i = 0$, then it is never selected (and so $\xi_i$ has no effect on the urn process). 

\bigskip 

\noindent{\bf Balanced in expectation, tenability, and extinction.} Throughout this article, we will say that an urn is {\em balanced in expectation} if 
\begin{equation}\label{eq:iidBalancedDef}
\text{for all } i=1,\ldots, q, \qquad \text{if } a_i > 0 \text{ then } \sum_{j=1}^q a_j \E \xi_{i,j} = b > 0
\end{equation}
for some constant $b$. Contrast with {\em balanced} urns, which is defined similarly but with the expectation removed in \eqref{eq:iidBalancedDef}. In this article, we are concerned with urns that are balanced in expectation, and show that many results for balanced urns hold for urns balanced in expectation.

We will say that $X_{n,i}$ is the `number of balls' of type $i$ at time $n$, though $X_{n,i}$ need not be an integer. We allow for $\xi_{i,j}$ to be negative (so balls may be subtracted from the urn), though we require the urn to be {\em tenable}, in the sense that $X_0$ and the distributions of the random vectors $\xi_1, \ldots, \xi_q$ must be defined such that the $X_n$ are non-negative at each step (and we cannot `remove' balls from the urn that do not exist). This can be accomplished for example by assuming that $\xi_1, \ldots, \xi_q$ satisfy a.s.
\[
\text{for } i \neq j, \qquad \xi_{i,j} \geq 0 \qquad \text{and} \qquad \xi_{j,i} \text{ is a multiple of } \xi_{i,i}. \]
(see for example \cite[Remark 4.2]{JANS:04}), though we do not make this assumption explicitly here. 

Define the event of {\em non-extinction up to time $n$} as
\begin{equation}\label{eq:ExtinctDef}
\Ext^c_n := \{ S_k > 0 \text{ for all } 0 \leq k \leq n \}.
\end{equation}
We say that the urn is non-extinct if it is non-extinct at time $n$ for all $n \geq 0$, that is, if $S_n > 0$ for all $n \geq 0$. From the definition of $S_n$, we see that the urn is non-extinct if there exist balls with nonzero activity; the urn may be extinct but still contain balls with activity zero. In the case of non-trivial balanced urns, the total activity $S_n$ increases by a constant non-zero amount at each step; that is, the urn is non-extinct at all times. Since normal limit laws were proved for urns with possible extinction \cite{JANS:04}, we will also consider such urns here; in fact one of our applications produces an urn with a positive probability of extinction (see Section \ref{sec:freezing}).

\bigskip

\noindent{\bf Intensity matrix and spectral decomposition.} A fundamental object for understanding limit results for P\'olya urns is the {\em intensity matrix} 
\begin{equation}\label{eq:ADef}
A := (a_j \E \xi_{j,i})_{i,j=1}^q.
\end{equation}
Note the indices of $A$; the $j$-th {\em column} is given by $a_j \E \xi_j$, the use of the transpose $A^T$ is used in some of the literature. 

Letting $\sigma(A)$ be the set of eigenvalues of $A$, there exists a set of projections $\{P_{\lambda} :\: \lambda \in \sigma(A)\}$ such that
\begin{equation}\label{eq:ProjDec}
\sum_{\lambda \in \sigma(A)} P_\lambda = I,
\end{equation}
\begin{equation}\label{eq:ProjNil}
AP_\lambda = P_\lambda A = \lambda P_\lambda + N_\lambda,
\end{equation}
where $N_\lambda = P_\lambda N_\lambda = N_\lambda P_\lambda$ is nilpotent. Also, $P_\lambda P_\mu = 0$ whenever $\lambda \neq \mu$, while $N_\lambda^{\nu_\lambda} \neq 0$ and $N_\lambda^{\nu_\lambda + 1} = 0$, where $\nu_\lambda + 1$ is the size of the largest Jordan block in the Jordan decomposition of $A$ with $\lambda$ on the diagonals. Thus we see that $\nu_\lambda = 0$ if $\lambda$ is simple. 

In the case of urns balanced in expectation, we have that $b$ is an eigenvalue of $A$ and that $a^T$ is a left eigenvector of $A$ associated with $b$, which follows by \eqref{eq:iidBalancedDef} and \eqref{eq:ADef} since 
\[ a^T A = \left( \sum_{i=1}^q a_i a_j \E \xi_{j,i}\right)_{j=1}^q = \left(a_j \E(a \cdot \xi_j)\right)_{j=1}^q = b a^T.\]
We will also assume throughout that $b \geq \Re \lambda$ for all other eigenvalues $\lambda$ of $A$. Much like in the case of balanced urns, this assumption is rather weak. In particular, this assumption is satisfied whenever the urn satisfies the assumptions (A1)--(A6) of \cite{JANS:04} and is balanced in expectation (see \cite[Lemma 5.4]{JANS:04}). 

The eigenvalues of $A$ are labelled $\lambda_1, \lambda_2, \ldots, \lambda_q$ in decreasing order of their real parts and Jordan block size, that is, $\Re \lambda_1 \geq \Re \lambda_2 \geq \cdots \geq \Re \lambda_q$ and if $\Re \lambda_i = \Re \lambda_j$ and $\nu_{\lambda_i} > \nu_{\lambda_j}$, then $i > j$. By assumption we have $b = \lambda_1$ is real. We say that the urn is {\em large} if $\lambda_1 < 2 \lambda_2$, {\em small} if $\lambda_1 \geq 2 \Re \lambda_2$, and {\em strictly small} if $\lambda_1 > 2 \Re \lambda_2$. We further make use of the following Lemma from \cite{JANS:25} that also holds for urns balanced in expectation:
\begin{lemma}{\cite[Lemma 2.3]{JANS:25}}\label{lem:lambda1}
Suppose $\lambda_1 = b$ is simple. Then there exists a unique right eigenvector $v_1$ of $A$ such that 
\[ a \cdot v_1 = 1.\]
Furthermore, the projection $P_{\lambda_1}$ is given by $P_{\lambda_1} = v_1 a^T$ and as a consequence for any $v \in \mathbb{C}^q$, 
\[ P_{\lambda_1} v = (a \cdot v) v_1.\]
\end{lemma}

\section{Main Results}\label{sec:results}

For simplicity of notation, we will use $C$ to denote an unspecified constant which is possibly different at each occurrence, and which depends only on the urn; that is, it depends only on $q$, the activity vector $a$, the initial composition $X_0$, and the distribution of $\xi_1, \ldots, \xi_q$. If the constant also depends on the $p$-norm, it is denoted by $C_p$. In particular, throughout, the constants $C$ and $C_p$ {\em do not} depend on $n$.

\medskip

For $p \geq 1$, denote the following statement
\begin{itemize}
\item[$\U{p}$:] the urn is tenable and balanced in expectation, $\lambda_1 = b > 0$, and $\left\lVert \xi_i \right\rVert_p < \infty$ for all $i=1, \ldots q$,
\end{itemize}
and recall the event $\Ext^c_n$ of non-extinction up to time $n$ \eqref{eq:ExtinctDef}. We first state the following general bound for the moments.
\begin{theorem}\label{thm:1}
Assume $\U{p}$ for some $p \geq 2$. Then for all $n \geq 2$, 
\begin{equation}\label{eq:thm1}
\left\Vert\left(X_n - \E [X_n| \Ext^c_n]\right)\1_{\Ext^c_n}\right\rVert_p \leq 
\begin{cases}
C_p n^{1/2} & \Re \lambda_2 < \lambda_1/2 \\
C_p n^{1/2}(\log n)^{\nu_2 + 1} & \Re \lambda_2 = \lambda_1/2 \\
C_p n^{\Re \lambda_2/\lambda_1}(\log{n})^{\nu_2} & \Re \lambda_2 > \lambda_1/2.
\end{cases}
\end{equation}
\end{theorem}

 We note that in the case $\Re \lambda_2 = \lambda_1 /2$ of \eqref{eq:thm1}, we were not able to achieve a bound that matches the normalizing constant for well-known normal limit laws (for example from \cite[Theorem 3.23]{JANS:04}). However, in the case $\Re \lambda_2 < \lambda_1/2$ of \eqref{eq:thm1}, the bound $C\sqrt{n}$ does match the normalizing constant for well-known normal limit laws, in particular in the case where conditioned on non-extinction,
 \begin{equation}\label{eq:CLT2}
\frac{X_n - \E\left[ X_n | \Ext^c_n\right]}{\sqrt{n}} \xrightarrow{d} \mathcal{N}(0,\Sigma)
\end{equation}
for some covariance matrix $\Sigma$.

\begin{theorem}\label{thm:MainThm}
Assume $\U{p}$ for all $p \geq 2$. Assume further that the urn is strictly small and that conditioned on non-extinction, the normal limit law \eqref{eq:CLT2} holds. Then \eqref{eq:CLT2} holds with convergence of all conditional moments. In particular, the following convergence of the covariance matrix holds: 
\[\mathrm{Cov}\left[X_n | \Ext^c_n \right]/n \xrightarrow{n \to \infty} \Sigma.\] 
\end{theorem}

For the expectation, we are able to prove the following asymptotics.

\begin{theorem}\label{thm:exp}
Assume $\U{2}$. If $\lambda_1 = b > \Re \lambda_2$, then
\[ \E\left[ X_n | \Ext^c_n \right] = n\lambda_1 v_1 + o(n).\]
If additionally, the urn is strictly small ($\lambda_1 = b > 2 \Re \lambda_2$), then
\[ \E\left[ X_n | \Ext^c_n \right]= n\lambda_1 v_1 + O(\sqrt{n}).\]
\end{theorem}
Supposing some technical assumptions and that the urn is strictly small, it is proved in \cite[Theorem 3.22]{JANS:04} that conditioned on non-extinction,
\begin{equation}\label{eq:CLT1}
\frac{X_n - n\lambda_1 v_1}{\sqrt{n}} \xrightarrow{d} \mathcal{N}(0, \Sigma)
\end{equation}
for the vector $v_1$ from Lemma \ref{lem:lambda1} and for a matrix $\Sigma$ given by \cite[(3.19)]{JANS:04}. To apply Theorem \ref{thm:MainThm}, we need to replace $n\lambda_1 v_1$ with $\E\left[ X_n | \Ext^c_n\right]$ in \eqref{eq:CLT1}, which is guaranteed by the following theorem.

\begin{theorem}\label{prop}
Assume $\U{2}$. If the urn is strictly small and conditioned on non-extinction admits the convergence \eqref{eq:CLT1}, then
\[ \E\left[ X_n | \Ext^c_n\right] = n\lambda v_1 + o(\sqrt{n}),\]
and so conditioned on non-extinction, the distributional convergence \eqref{eq:CLT2} also holds.
\end{theorem}

\begin{remark}
In applications, if $\U{p}$ holds for all $p \geq 2$ and one can prove the convergence \eqref{eq:CLT1} holds (for example by applying \cite[Theorem 3.22]{JANS:04}), Theorem \ref{prop} along with Theorem \ref{thm:MainThm} implies that \eqref{eq:CLT1} and \eqref{eq:CLT2} converge to the same limit, with convergence of all moments holding as well. This is the approach we use in Section \ref{sec:applications}.
\end{remark}

\begin{remark}\label{rmk:5.4}
We note some minor mistakes in \cite{JANS:04} relevant to our results. 

The impossibility of non-extinction stated in \cite[Lemma 2.1]{JANS:04} only holds if $\sum_j \xi_{i,j} \geq 0$ for all $i$ and $\sum_j \E\xi_{i,j} > 0$ for some $i$. (i.e., by removing the expectations in the first inequality). Only then does the number of balls in the urn never decrease (guaranteeing non-extinction). The remainder of the proof appears to be correct; that is, assuming that $A$ is irreducible, that (A1) and (A2) hold, and that $\sum_j \E\xi_{i,j} \geq 0$ for all $i$ and $\sum_j \E \xi_{i,j} > 0$ for some $i$, then assumptions (A1)--(A6) of \cite{JANS:04} hold. 

 We also note a minor mistake in the statement of \cite[Lemma 5.4]{JANS:04}. It is stated that if the urn is balanced in expectation, the matrix $\Sigma$ from \eqref{eq:CLT1} can be replaced with a matrix $b\Sigma_1$ admitting a simpler calculation. However, in the proof of the lemma, it is only assumed the urn is balanced (and not balanced in expectation), and in fact counter examples exist that show that this replacement is not true if we only assume balanced in expectation. The remaining statement of \cite[Lemma 5.4]{JANS:04} appears to be correct, and the replacement of $\Sigma$ with $b \Sigma_I$ does hold if the urn is balanced.  
\end{remark}

\section{Bounding spectral projections of moments}\label{sec:momentBounds}

The goal in this section is to prove the following theorem, the proof of which is at the end of this section. 

\begin{theorem}\label{thm:momentBounds}
Assume $\U{p}$ for some $p \geq 2$. The following hold:
\begin{itemize}
\item[(i)] For all eigenvalues $\lambda \in \sigma(A)$ and $n \geq 2$, 
\begin{equation}\label{eq:lambdaBound}
\left\lVert P_\lambda\left(X_n - \E\left[X_n |\Ext^c_n\right]\right)\1_{\Ext^c_n}\right\rVert_p \leq 
\begin{cases}
C_p n^{1/2} & \Re \lambda < \lambda_1/2 \\
C_p n^{1/2}(\log n)^{\nu_\lambda + 1} & \Re \lambda = \lambda_1/2 \\
C_p n^{\Re\lambda_2/\lambda_1}(\log{n})^{\nu_\lambda} & \Re \lambda > \lambda_1/2.
\end{cases}
\end{equation}
\item[(ii)] If $\lambda_1 = b$ is a simple eigenvalue, then $P_{\lambda_1}\left(X_n - \E\left[ X_n | \Ext^c_n\right]\right) = \left(S_n - \E\left[ S_n | \Ext^c_n\right]\right)v_1 $ and 
\begin{equation}\label{eq:lambda1Bound}
\left\lVert P_{\lambda_1}\left(X_n - \E\left[ X_n | \Ext^c_n\right]\right)  \1_{\Ext^c_n}\right\rVert_p \leq C_p \sqrt{n}.
\end{equation}
\end{itemize}
\end{theorem}

\bigskip

\noindent{\bf A note on extinction.} A method used to study P\'olya urns is to embed the process into a continuous time Markov branching process, notably in \cite{ATKA:68}, \cite[Sec. V.9]{ATNE:72}, and \cite{JANS:04}. The process $\mathcal{X}(t) = (\mathcal{X}_1(t), \ldots, \mathcal{X}_q(t))$ is defined according to the same parameters $a_1, \ldots, a_q$ and $\xi_1, \ldots, \xi_q$. A particle of type $i$ lives for an exponentially distributed time with mean $a_i^{-1}$, independently of all other particles. After it dies, the particle is replaced with a number of particles distributed as $(\xi_i + \delta_{i,j})_{j=1}^q$ independently from everything that has occurred so far. Each new particle of type $j$ is assigned a new lifespan exponentially distributed with mean $a_j^{-1}$. The process is defined to be right-continuous. Define $\tau_n$ to be the time of the $n$-th birth event; then $(\mathcal{X}(\tau_n))_{n=0}^\infty$ is equal in distribution to $(X_n)_{n=0}^\infty$. As is well-known, the condition $\lambda_1 = b > 0$ guarantees that $\mathcal{X}(t)$ has positive probability of non-extinction, meaning that $\mathcal{X}(t)$ contains particles with positive activity for all $t \in \R$ \cite[Sec. V.7]{ATNE:72}.

From the defined embedding, we see that the probability of non-extinction is positive for $(X_n)_{n=0}^\infty$ given our assumption that $\lambda_1 = b > 0$, see also the discussion contained in \cite[Sec. 2]{JANS:04}. We have also defined $(\Ext_n^c)_{n=0}^\infty$ to be a nested sequence of sets (with $\Ext_{n+1}^c \subseteq \Ext_{n}^c$), and so
\begin{equation}\label{eq:ExtinctionProb}
\lim_{n \to \infty} \E \left[\1_{\Ext_n^c}\right] = \lim_{n \to \infty} \P\left(\Ext_n^c\right) =  \P\left(\bigcap_{n=0}^\infty \Ext_n^c\right) > 0.
\end{equation}
From our definition of non-extinction \eqref{eq:ExtinctDef} and from the definition of $S_n$ \eqref{eq:SDef}, an extinct urn may still contain balls with activity zero. Since the columns of $A$ corresponding to balls with activity zero are also zero, we see that 
\begin{equation}\label{eq:aExtinct}
S_n \1_{\Ext_n^c} = a \cdot X_n \1_{\Ext_n^c} = a \cdot X_n = S_n, \qquad \text{and} \qquad AX_n \1_{\Ext_n^c} = A X_n.
\end{equation}

\bigskip

\noindent{\bf Decomposition of $X_n$.} We will be evaluating martingales and martingale difference sequence, and so will be encountering conditional expectations with respect to $\sigma$-fields. To avoid confusion for the remainder of this section up to Section \ref{sec:proofs}, we will denote the expectation of a random variable $X$ conditioned on the {\em event} of non-extinction up to time $n$ by 
\begin{equation}\label{eq:CondExpDef}
\E_{\Ext_n^c}\left[ X \right] := \E\left[ X | \Ext_n^c\right] = \frac{\E\left[ X \1_{\Ext_n^c}\right]}{\E\left[\1_{\Ext_n^c}\right]},\end{equation}
and reserve the notation $\E\left[X | \mathcal{F}\right]$ for conditional expectation with respect to the {\em $\sigma$-field} $\mathcal{F}$. From \eqref{eq:ExtinctionProb}, we see that \eqref{eq:CondExpDef} is well-defined since the denominator is never zero. 

\medskip

Throughout the following discussion before the statement of lemmas, assume $\U{p}$ holds for some $p \geq 2$. Let $\mathcal{F}_n$ be the $\sigma$-field generated by $X_1, \ldots, X_n$. Let $I_n$ be the colour of the $n$-th drawn ball if $S_{n-1} > 0$. Recall that no ball is selected if $S_{n-1} = 0$ (i.e., if the urn is extinct), and so for all $n$,
\begin{equation}\label{eq:SelectNonExt}
\sum_{j=1}^q \P(I_{n+1} = j | \mathcal{F}_n) = \1_{\Ext^c_n}.
\end{equation}
Then by \eqref{eq:BallProbab}
\[ \P(I_{n+1} = j | \mathcal{F}_n) =\frac{a_j X_{nj}}{S_n}\1_{\Ext^c_n},\]
and recalling $\Delta X_n \sim \xi_j$ if the $j$'th ball is selected,
\begin{equation}\label{eq:DeltaFn} 
\E\left[\Delta X_n| \mathcal{F}_n\right] = \sum_{j=1}^q \P(I_{n+1} = j | \mathcal{F}_n)\E\xi_j = \frac{\1_{\Ext^c_n}}{S_n}\sum_{j=1}^q a_j X_{n,j} \E \xi_j = \frac{1}{S_n}A X_n\1_{\Ext^c_n},
\end{equation}
while by assumption $\U{p}$,
\begin{equation}\label{eq:ExpBoundDelta}
\E\left[|\Delta X_n |^p | \mathcal{F}_n\right] = \sum_{j=1}^q \P(I_{n+1} = j | \mathcal{F}_n) \E | \xi_j|^p \leq C_p \1_{\Ext^c_n},
\end{equation}
and so 
\begin{equation}\label{eq:BoundDelta}
\left\lVert \Delta X_n \right\rVert_p \leq C_p.
\end{equation}
Recalling the vector of activities $a$, \eqref{eq:iidBalancedDef}, \eqref{eq:SelectNonExt} and \eqref{eq:DeltaFn} yield
\begin{equation}\label{eq:adotDelta}
\E\left[ a \cdot \Delta X_n | \mathcal{F}_n\right] = \sum_{j=1}^q \P\left(I_{n+1} = j | \mathcal{F}_n\right) a \cdot \E \xi_j = b \1_{\Ext_n^c}, 
\end{equation}
and so by \eqref{eq:aExtinct}, $\E \left[ a \cdot \Delta X_n\right] = \E\left[ a \cdot \Delta X_n \1_{\Ext_n^c}\right] = \E\left[ b \1_{ \Ext_n^c}\right]$. 
Conditioning on non-extinction, 
\begin{equation}\label{eq:CondadotDelta}
\E_{\Ext_n^c} \left[ a \cdot \Delta X_n\right] = \frac{\E\left[a \cdot \Delta X_n \1_{ \Ext_n^c}\right]}{\E \left[\1_{\Ext_n^c}\right]} = b.
\end{equation}
From \eqref{eq:SDef} and \eqref{eq:update}, 
\begin{equation}\label{eq:SnasSum}
S_n = a \cdot \left( X_0 + \sum_{\ell=0}^{n-1} \Delta X_\ell\right) = a \cdot X_0 + \sum_{\ell=0}^{n-1} a \cdot \Delta X_\ell,
\end{equation}
which we use to find the expected total activity of the urn conditioned on non-extinction at time $n$: 
\begin{equation}\label{eq:omegaDef}
\omega_n := \E_{\Ext_n^c} \left[ S_n \right] = a \cdot X_0 + \sum_{\ell=0}^{n-1} \E_{\Ext_n^c}  \left[ a \cdot \Delta X_\ell \right] = a \cdot X_0 + n b.
\end{equation}
Define 
\begin{align}
Y_n &:= \Delta X_{n-1} - \E\left[\Delta X_{n-1} | \mathcal{F}_{n-1}\right], \label{eq:YDef}\\
Z_n &:= \left(\frac{\omega_{n-1} - S_{n-1}}{\omega_{n-1}} \right)\E\left[\Delta X_{n-1} | \mathcal{F}_{n-1}\right]. \label{eq:ZDef}
\end{align}
Noting from \eqref{eq:aExtinct} that $AX_n = A X_n \1_{\Ext_n^c}$, we apply \eqref{eq:DeltaFn} and decompose $\Delta X_n$ as 
\[ \Delta X_n = Y_{n+1} + Z_{n+1} + \omega_n^{-1}A X_n.\]
By \eqref{eq:update} and induction we have 
\[X_n = X_{n-1} + \Delta X_{n-1}= \left(I + \omega_{n-1}^{-1} A\right) X_{n-1} + Y_n + Z_n = \prod_{k=0}^{n-1}\left(I + \omega_k^{-1}A\right) X_0 + \sum_{\ell = 1}^n \prod_{k=\ell}^{n-1}\left(I + \omega_k^{-1} A\right)\left(Y_\ell + Z_\ell\right).\]
By defining  
\begin{equation}\label{eq:FDef}
F_{i,j} := \prod_{i \leq k < j}(I+\omega_k^{-1} A),
\end{equation}
we are left with the decomposition of $X_n$ as 
\begin{equation}\label{eq:XDecomp}
X_n = F_{0,n} X_0 + \sum_{\ell = 1}^n F_{\ell,n}\left(Y_\ell + Z_\ell\right).
\end{equation}
As discussed in the introduction, the addition of the `noisy' parts $Z_\ell$ in \eqref{eq:XDecomp} differentiates the balanced in expectation case from the decomposition appearing in the balanced case of previous works.

For any $P_\lambda$, we therefore have 
\begin{equation}\label{eq:ProjXDecomp}
P_\lambda X_n = P_\lambda F_{0,n} X_0 + \sum_{\ell=1}^n P_\lambda F_{\ell,n} Y_\ell + \sum_{\ell=1}^n P_\lambda F_{\ell,n} Z_\ell.
\end{equation}
Since $P_\lambda F_{0,n} X_0$ is constant, $\E_{\Ext_n^c}\left[P_\lambda F_{0,n} X_0\right] = P_\lambda F_{0,n} X_0$. Thus, we see that 
\begin{equation}\label{eq:ProjX-EXDecomp}
P_\lambda\left(X_n - \E_{\Ext_n^c}\left[ X_n\right] \right) = \sum_{\ell=1}^n P_\lambda F_{\ell,n} Y_\ell - \sum_{\ell=1}^n P_\lambda F_{\ell,n} \E_{\Ext_n^c}\left[ Y_\ell\right] + \sum_{\ell=1}^n P_\lambda F_{\ell,n} Z_\ell  - \sum_{\ell=1}^n P_\lambda F_{\ell,n} \E_{\Ext_n^c} \left[ Z_\ell \right].
\end{equation}

\bigskip 

\noindent{\bf Bounding terms of the decomposition.} The definition of $F_{i,j}$ in \eqref{eq:FDef} is precisely the same as in \cite{JANS:20} and \cite{JANS:25}, and we use the following bound on the norms of $P_\lambda F_{i,j}$:
\begin{lemma}{\cite[Lemma 6.1]{JANS:25}}\label{lem:6.1}
For every eigenvalue $\lambda \in \sigma(A),$
\[ \left\lVert P_\lambda F_{\ell,n} \right\rVert \leq C\left(\frac{n}{\ell}\right)^{\Re \lambda/b} \left(1 + \log\left(\frac{n}{\ell}\right)\right)^{\nu_\lambda}, \qquad 1 \leq \ell \leq n < \infty.\]
\end{lemma}

The bounds for the terms in \eqref{eq:ProjXDecomp} containing $Y_n$ are treated exactly the same as in \cite{JANS:25}. 

\begin{lemma}\label{lem:YBound}
Assume $\U{p}$ for some $p \geq 2$. Then $(Y_n)_{n=1}^\infty$ is a martingale difference sequence with $\left\lVert Y_n \right\rVert_p < C_p$, and for all $\lambda \in \sigma(A)$ and $n \geq 2$, 
\[ \left\Vert \sum_{\ell = 1}^n P_\lambda F_{\ell,n} Y_\ell \right\Vert_p \leq 
\begin{cases}
C_p n^{1/2} & \Re \lambda < \lambda_1/2 \\
C_p n^{1/2}(\log n)^{\nu_\lambda + \frac{1}{2}} & \Re \lambda = \lambda_1/2 \\
C_p n^{\Re \lambda/\lambda_1}(\log{n})^{\nu_\lambda} & \Re \lambda > \lambda_1/2.
\end{cases}\]
\end{lemma}

\begin{proof}
By definition \eqref{eq:YDef} we see that $Y_n$ is $\mathcal{F}_n$-measurable and $\E\left[Y_n | \mathcal{F}_{n-1}\right] = 0$, and so $(Y_n)_{n=1}^\infty$ is a martingale difference sequence, while $\left\lVert Y_n \right\rVert_p \leq C_p$ follows from \eqref{eq:BoundDelta}. The result then follows from Lemma \cite[Lemma 5.1]{JANS:25} and \cite[Lemma 6.2]{JANS:25}.
\end{proof}

To both prove Theorem \ref{thm:momentBounds} (ii) and to bound the terms in \eqref{eq:ProjXDecomp} containing $Z_n$, we bound the $L^p$-norm of $(S_n - \omega_n)\1_{\Ext_n^c}$.

\begin{lemma}\label{lem:Burk}
Assume $\U{p}$ holds for some $p \geq 2$. Then for all $n \geq 2$, 
\[ \left\lVert (S_n - \omega_n) \1_{\Ext^c_n} \right\rVert_p \leq C_p \sqrt{n}.\]
\end{lemma}

\begin{proof}
Define the sequence of random variables $(W_n)_{n=1}^\infty$ by 
\begin{equation}\label{eq:Burk1}
W_n := a \cdot Y_n = \left(a \cdot \Delta X_{n-1} - b \right)\1_{\Ext_{n-1}^c},
\end{equation}
where the second equation follows \eqref{eq:adotDelta} and \eqref{eq:YDef}. 
From Lemma \ref{lem:YBound}, we see that $(W_n)_{n=1}^\infty$ is a martingale difference sequence with $\left\lVert W_n \right\rVert_p \leq C_p$. Then $\sum_{\ell = 1}^n W_n$ is a martingale, and Burkholder's inequality \cite[Theorem 9]{BURK:66} (see also \cite[Theorem 10.9.5]{GUT}) implies the existence of a constant $C_p$ depending only on $p$ such that 
\begin{equation}\label{eq:Burk2}
\left\lVert \sum_{\ell = 1}^n W_\ell \right\rVert_p \leq C_p \left\lVert \left( \sum_{\ell=1}^n W_\ell^2 \right)^{1/2}\right\rVert_p.\end{equation}
From Minkowski's inequality,
\begin{equation}\label{eq:Burk3}
\left\lVert \left( \sum_{\ell=1}^n W_\ell^2 \right)^{1/2}\right\rVert_p^2 = \left\lVert \sum_{\ell=1}^n W_\ell^2 \right\rVert_{p/2} \leq \sum_{\ell=1}^n \lVert W_\ell^2 \rVert_{p/2} = \sum_{\ell=1}^n \lVert W_\ell \rVert_p^2.
\end{equation}
From \eqref{eq:BoundDelta}, \eqref{eq:YDef}, and \eqref{eq:Burk1}, $\left\Vert W_n \right\rVert_p \leq C_p$, and so by \eqref{eq:Burk2} and \eqref{eq:Burk3}, 
\begin{equation}\label{eq:SnMartBound}
\left\lVert \sum_{\ell = 1}^n W_\ell \right\rVert_p \leq C_p\sqrt{n}.
\end{equation}
Since $\Ext_n^c \subseteq \Ext_\ell^c$ for all $\ell \leq n$, it follows from \eqref{eq:SnasSum} and \eqref{eq:omegaDef} that 
\begin{align*}
\left(S_n - \omega_n\right)\1_{\Ext_n^c} &= \left( \left(a \cdot X_0 + \sum_{\ell = 0}^{n-1} a \cdot \Delta X_\ell\right) - \left(a \cdot X_0 + n b\right)\right)\1_{\Ext_n^c} \\
&= \left(\sum_{\ell=0}^{n-1}\left(a \cdot \Delta X_\ell - b \right)\1_{\Ext_\ell^c}\right)\1_{\Ext_n^c} \\
&= \left(\sum_{\ell =1}^n W_\ell \right)\1_{\Ext_n^c}.
\end{align*}
The result now follows from \eqref{eq:SnMartBound} since 
\[ \left\lVert (S_n - \omega_n)\1_{\Ext_n^c}\right\rVert_p = \left\lVert \left(\sum_{\ell=1}^n W_\ell \right)\1_{\Ext_n^c}\right\rVert_p \leq \left\lVert \sum_{\ell=1}^n W_\ell \right\rVert_p \leq C_p\sqrt{n}.\]
\end{proof}

Finally, we prove the bounds for the terms of \eqref{eq:ProjXDecomp} containing $Z_n$.

\begin{lemma}\label{lem:ZBound}
Assume $\U{p}$ for some $p \geq 2$. For all $\lambda \in \sigma(A)$ and $n \geq 2$,  
\[ \left\Vert \sum_{\ell=1}^n P_\lambda F_{\ell,n} Z_\ell \right\Vert_p \leq 
\begin{cases}
C_p n^{1/2} & \Re \lambda < \lambda_1/2 \\
C_p n^{1/2}(\log n)^{1+\nu_\lambda} & \Re \lambda = \lambda_1/2 \\
C_p n^{\Re \lambda/\lambda_1}(\log{n})^{\nu_\lambda} & \Re \lambda > \lambda_1/2.
\end{cases}\]
\end{lemma}

\begin{proof}
First we show that 
\begin{equation}\label{eq:ZBound}
\left\lVert Z_n \right\rVert_p \leq \frac{C_p}{\sqrt{n}}.
\end{equation}
Since $\omega_n^{-1}(\omega_n - S_n)$ is $\mathcal{F}_n$-measurable, we apply \eqref{eq:ExpBoundDelta} and get
\begin{equation}\label{eq:ZProof1}
\E\left[ |Z_n|^p | \mathcal{F}_n\right] = \E\left[\left.\left(\frac{\omega_n - S_n}{\omega_n}\E\left[\Delta X_n | \mathcal{F}_{n}\right]\right)^p \right\rvert \mathcal{F}_n\right] \leq  \left(\frac{\omega_n - S_n}{\omega_n}\right)^p \E\left[ | \Delta X_n |^p | \mathcal{F}_{n}\right] \leq C_p\1_{\Ext^c_n}\left(\frac{\omega_n - S_n}{\omega_n}\right)^p.
\end{equation}
Then \eqref{eq:ZBound} follows by Lemma \ref{lem:Burk}, \eqref{eq:omegaDef}, and \eqref{eq:ZProof1}.

Let $\gamma := \text{Re} \lambda /b$. First suppose $\gamma < 1/2$. From \eqref{eq:ZBound}, Lemma \ref{lem:6.1}, and Minkowski's inequality, 
\begin{multline*}
\left\Vert \sum_{\ell=1}^n P_\lambda F_{\ell,n} Z_\ell \right \Vert_p \leq \sum_{\ell=1}^n ||P_\lambda F_{\ell,n}|| \, ||Z_\ell||_p \leq \frac{C_p}{\sqrt{n}}\sum_{\ell =1}^n \left(\frac{\ell}{n}\right)^{-\gamma-1/2}\left(1 + \log \frac{n}{\ell}\right)^{\nu_\lambda} \\
\leq C_p\sqrt{n} \int_0^1 x^{-\gamma-1/2}\left(1 + \log \frac{1}{x}\right)^{\nu_\lambda}dx = C_p\sqrt{n}.
\end{multline*}
Similarly if $\gamma = 1/2$, then 
\[ \left\Vert \sum_{\ell=1}^n P_\lambda F_{\ell,n} Z_\ell \right \Vert_p \leq C_pn^{1/2}\sum_{\ell =1}^n \ell^{-1}\left(1 + \log \frac{n}{\ell}\right)^{\nu_\lambda} \leq C_p n^{1/2}(\log{n})^{1 + \nu_\lambda}.\]
Finally, if $\gamma > 1/2$,
\[ \left\Vert \sum_{\ell=1}^n P_\lambda F_{\ell,n} Z_\ell \right \Vert_p \leq C_pn^{\gamma}\sum_{\ell=1}^n \ell^{-\gamma - 1/2}(\log n)^{\nu_\lambda}.\]
\end{proof}

\bigskip

\noindent{\bf Proof of Theorem \ref{thm:momentBounds}.} 
To start, we note as an easy application of Jensen's inequality that 
\[
\left\lVert \sum_{ \ell =1}^n P_\lambda F_{\ell,n} \E \left[ Y_\ell \1_{\Ext_n^c} \right]\right\Vert_p \leq \left\lVert \sum_{ \ell =1}^n P_\lambda F_{\ell,n} \E Y_\ell \right\Vert_p \leq \left\lVert \sum_{ \ell =1}^n P_\lambda F_{\ell,n} Y_\ell \right\Vert_p,
\]
which also holds with the $Y_\ell$ terms replaced by $Z_\ell$ by the same argument. We then apply \eqref{eq:ExtinctionProb} to bound $\E \1_{\Ext_n^c}$ uniformly in $n$ and have that 
\begin{equation}\label{eq:BoundConditionalYZ}
\left\lVert \sum_{ \ell =1}^n P_\lambda F_{\ell,n} \E_{\Ext_n^c} [Y_\ell] \right\Vert_p \leq C \left\lVert \sum_{ \ell =1}^n P_\lambda F_{\ell,n} Y_\ell \right\Vert_p \qquad \text{and} \qquad  \left\lVert \sum_{ \ell =1}^n P_\lambda F_{\ell,n} \E_{\Ext_n^c} [Z_\ell] \right\Vert_p \leq C \left\lVert \sum_{ \ell =1}^n P_\lambda F_{\ell,n} Z_\ell \right\Vert_p.
\end{equation}
Theorem \ref{thm:momentBounds}(i) now holds from \eqref{eq:ProjX-EXDecomp}, Minkowski's inequality, \eqref{eq:BoundConditionalYZ}, and applying Lemmas \ref{lem:YBound} and \ref{lem:ZBound}.\\

For Theorem \ref{thm:momentBounds}(ii), we apply Lemma \ref{lem:lambda1} which along with \eqref{eq:SDef} and \eqref{eq:omegaDef} yields
\begin{equation}\label{eq:4.1proof1} P_{\lambda_1}(X_n - \E_{\Ext_n^c} X_n) \1_{\Ext^c_n} = a \cdot (X_n - \E_{\Ext_n^c} X_n) v_1 \1_{\Ext^c_n} = (S_n - \E_{\Ext_n^c} S_n)v_1 \1_{\Ext^c_n} = (S_n - \omega_n)v_1 \1_{\Ext_n^c}.
\end{equation}
The result now follows from \eqref{eq:4.1proof1} and Lemma \ref{lem:Burk}. 
\qed

\section{Proof of main results}\label{sec:proofs}

\noindent {\bf Proof of Theorem \ref{thm:1}.}
First suppose $\Re \lambda_2 = \lambda_1 = b$. In this case the result follows from Minkowski's inequality, 
\[ \left\lVert X_n \right\rVert_p \leq \left\lVert X_0 \right\rVert_p + \sum_{k=0}^{n-1}\left\lVert \Delta X_k \right\rVert_p \leq C_p + C_p n \leq C_p n.\]
By Jensen's inequality and \eqref{eq:ExtinctionProb}, $\left\lVert \E \left[X_n | \Ext_n^c\right] \right\rVert_p \leq C \left\lVert \E X_n \1_{\Ext_n^c}\right\rVert \leq \left\lVert X_n \right\rVert_p$ and so Minkowski's inequality once more yields $\left\lVert (X_n - \E \left[X_n | \Ext_n^c\right])\1_{\Ext_n^c} \right\rVert_p \leq C_p n$. \\

Now assume that $\Re \lambda_2 < \lambda_1$, and so $\lambda_1$ is simple. From Theorem \ref{thm:momentBounds}(ii), we have that 
\[\left\lVert P_{\lambda_1}\left( X_n - \E \left[ X_n | \Ext_n^c \right] \right)\1_{\Ext_n^c} \right\rVert_p \leq C_p\sqrt{n}.\]
Thus from the decomposition \eqref{eq:ProjDec} and Minkowski's inequality, 
\[ \left\lVert\left( X_n - \E \left[ X_n | \Ext_n^c \right] \right) \1_{\Ext_n^c} \right\rVert_p \leq \sum_{\lambda \in \sigma(A)} \left\lVert P_\lambda\left(X_n - \E \left[ X_n | \Ext_n^c \right]\right)\1_{\Ext_n^c}  \right\rVert_p \leq C_p\sqrt{n} + \sum_{\lambda \neq \lambda_1}\left\lVert P_\lambda\left(X_n - \E \left[ X_n | \Ext_n^c \right]\right))\1_{\Ext_n^c}  \right\rVert_p\]
is dominated by the bound for $\lambda = \lambda_2$ from Theorem \ref{thm:momentBounds}(i), completing the proof. \qed

\bigskip

\noindent{\bf Proof of Theorem \ref{thm:MainThm}.}
As is well-known, for $1 \leq p_1 < p_2$, if non-negative random variables $M_1, M_2, \ldots$ satisfy $\sup_n \E M_n^{p_2} < \infty$, then $\{ M_n^{p_1}\}_{n=1}^\infty$ is uniformly integrable (see for example \cite[Theorem 5.4.2]{GUT}). Thus by assuming $\U{p}$ for all $p \geq 2$, Theorem \ref{thm:1} therefore implies that 
\[ \left\{ \left\lvert\frac{X_n - \E \left[ X_n | \Ext_n^c \right]}{\sqrt{n}}\right\rvert^p\1_{\Ext_n^c} \right\}_{n=2}^\infty\]
is uniformly integrable for all $p \geq 1$. The convergence in distribution \eqref{eq:CLT2} then implies the convergence of moments (see for example \cite[Theorem 5.5.9 (i)]{GUT}). \qed

\bigskip

\noindent{\bf Proof of Theorem \ref{thm:exp}.} Recalling the definition of $S_n$ from \eqref{eq:SDef}, then Lemma \ref{lem:lambda1} and \eqref{eq:omegaDef} imply
\begin{equation}\label{eq:expLambda1}
\E\left[P_{\lambda_1}X_n | \Ext_n^c\right] = \E\left[ (a \cdot X_n) v_1 | \Ext_n^c\right] = \E \left[ S_n | \Ext_n^c\right]v_1 = (\omega_0 + nb)v_1 = (\omega_0 + n\lambda_1)v_1.
\end{equation}
Let $\lambda \neq \lambda_1$, and so $\Re \lambda <  \lambda_1$. By the decomposition \eqref{eq:ProjXDecomp}, 
\begin{equation}\label{eq:expLambda}
\E\left[ P_\lambda X_n | \Ext_n^c\right] =  P_\lambda F_{0,n} X_0 + \sum_{\ell=1}^n P_\lambda F_{\ell,n} \E\left[Y_\ell | \Ext_n^c\right] + \sum_{\ell=1}^n P_\lambda F_{\ell,n} \E\left[Z_\ell| \Ext_n^c\right].
\end{equation}
From \eqref{eq:ProjNil} (and since $N_\lambda P_\lambda = N_\lambda$), and the definition of $F_{0,n}$ from \eqref{eq:FDef}, 
\[ P_\lambda F_{0,n} = P_\lambda\left(I + \omega_0^{-1}A\right)F_{1,n} = \left(P_\lambda + \omega_0^{-1}\left(\lambda P_\lambda + N_\lambda\right)\right) F_{1,n} = \left(I + \omega_0^{-1}(\lambda I + N_\lambda)\right)P_\lambda F_{1,n}.\]
Therefore, by Lemma \ref{lem:6.1}, 
\begin{equation}\label{eq:ExpLambdaX0}
\left\lvert P_\lambda F_{0,n} X_0 \right\rvert  \leq \left\lVert I + \omega_0^{-1}(\lambda I + N_\lambda)\right\rVert \, \left\lVert P_\lambda F_{1,n} \right\rVert \left\lvert X_0 \right\rvert \leq C n^{\Re \lambda / \lambda_1}\left(1 + \log n \right)^{\nu_\lambda}.
\end{equation}
Since we assume $\U{2}$ holds, an easy application of Jensen's inequality with \eqref{eq:ExtinctionProb} implies 
\begin{equation}\label{eq:ExpLambdaYZ}
\begin{aligned}
\left\lvert \sum_{\ell = 1}^n P_\lambda F_{\ell,n} \E\left[ Y_\ell | \Ext_n^c\right]\right\rvert &\leq C \left\lVert \sum_{\ell = 1}^n P_\lambda F_{\ell,n} Y_\ell\right\rVert_2, \\ \\
\left\lvert \sum_{\ell = 1}^n P_\lambda F_{\ell,n} \E\left[ Z_\ell | \Ext_n^c\right]\right\rvert &\leq C \left\lVert \sum_{\ell = 1}^n P_\lambda F_{\ell,n} Z_\ell\right\rVert_2.
\end{aligned}
\end{equation}
From the decomposition \eqref{eq:ProjDec}, 
\[ \E\left[ X_n | \Ext_n^c\right] = \E \left[ P_{\lambda_1} X_n | \Ext_n^c\right] + \sum_{\lambda \neq \lambda_1} \E \left[ P_\lambda X_n | \Ext_n^c\right],\]
and so the result follows from \eqref{eq:expLambda1} and \eqref{eq:expLambda}, along with \eqref{eq:ExpLambdaX0}, \eqref{eq:ExpLambdaYZ}, and Lemmas \ref{lem:YBound} and \ref{lem:ZBound}.
\qed

\bigskip

\noindent{\bf Proof of Theorem \ref{prop}.} By Theorem \ref{thm:1}, Theorem \ref{thm:exp}, and Minkowski's inequality, 
\[ \left\lVert \left(X_n - n\lambda_1 v_1\right)\1_{\Ext_n^c} \right\rVert_2 \leq \left\lVert \left(X_n - \E \left[X_n | \Ext_n^c\right]\right)\1_{\Ext_n^c}\right\rVert_2 + \left\lvert \left(\E\left[ X_n | \Ext_n^c\right] - n\lambda_1 v_1 \right)\1_{\Ext_n^c}\right\rvert \leq C \sqrt{n}\]
uniformly for all $n$, which implies that the sequence 
\[ \left\{ \left\lvert \frac{ X_n - n\lambda_1 v_1}{\sqrt{n}}\right\rvert \1_{\Ext_n^c}\right\}_{n=0}^{\infty}\]
is uniformly integrable (once more see \cite[Theorem 5.5.7]{GUT}). This uniform integrability along with the assumed convergence in distribution conditioned on non-extinction \eqref{eq:CLT1} implies 
\[ \E\left[\left. \frac{ X_n - n\lambda_1 v_1}{\sqrt{n}}\right\rvert \Ext_n^c\right] \xrightarrow{n \to \infty} \E\, \mathcal{N}(0, \Sigma) = 0\]
(see for example \cite[Theorem 5.5.9]{GUT}), which therefore implies 
\[ \E \left[X_n | \Ext_n^c\right] = n\lambda_1 v_1 + o(\sqrt{n})\]
and so conditioned on non-extinction,
\[ \frac{X_n - \E \left[X_n| \Ext_n^c\right]}{\sqrt{n}} = \frac{X_n - n \lambda_1 v_1}{\sqrt{n}} - \frac{\E\left[ X_n | \Ext_n^c\right] - n\lambda_1 v_1}{\sqrt{n}} \xrightarrow{d} \mathcal{N}(0, \Sigma) + 0.\]
\qed

\section{Applications}\label{sec:applications}

As mentioned in the introduction, generalized P\'olya urns have long been used to prove normal limit laws for the degree distributions of increasing trees, and in most cases, the associated urns are balanced and so also admit convergence of all moments thanks to previous results on balanced urns. Here we highlight two examples where we may apply our results to prove new results for the convergence of moments for degree distributions of certain recent random graph models similar to increasing trees. 

\subsection{Uniform attachment with freezing}\label{sec:freezing}

Uniform attachment trees with freezing are a class of random trees of recent interest (see for example \cite{BBKK:25, BBKK2:25, BBCKK:25, KAKS:25}). In the general case, we are given an infinite sequence $\bm{x} = (x_i)_{i=1}^\infty$ where $x_i \in \{-1,1\}$. The initial tree $\mathcal{T}_0(\bm{x})$ consists of a single vertex with the label $a$ (for {\em active}). At each step $n \geq 1$, the tree $\mathcal{T}_n(\bm{x})$ is constructed from $\mathcal{T}_{n-1}(\bm{x})$ as follows: if there are no active vertices (labelled $a$) in $\mathcal{T}_{n-1}(\bm{x})$, then $\mathcal{T}_n(\bm{x}) = \mathcal{T}_{n-1}(\bm{x})$. Otherwise, a vertex $v_n$ is selected uniformly at random among the vertices with label $a$. If $x_n = -1$, then $v_n$ becomes {\em frozen} and is given the label $n$, and if $x_n = +1$, then a edge is drawn from $v_n$ to a new active vertex with label $a$. 

When $x_i = +1$ for all $i \geq 1$, the trees $\mathcal{T}_n(\bm{x})$ are distributed as random recursive trees. Here we consider the case where the $x_i$ are i.i.d. with $\P(x_i = +1) = p$ and $\P(x_i = -1) = 1-p$. As previously shown \cite[Theorem 2]{BBKK:25}, in this case the limiting tree $\mathcal{T}_{\infty}(\bm{x})$ is distributed as a Bienaym\'e tree with Geometric $\text{Ge}(1-p)$ offspring distribution. 

\bigskip

\noindent{\bf A P\'olya urn for the vertex degrees.} For any fixed $K \geq 1$, we construct an urn where balls represent vertices with outdegree $0, \ldots, K-1$; the outdegree of a vertex is the number of children. Our urn has $2K+2$ types of balls; for $m=0, \ldots, K-1$, balls of type $2m+1$ represent active vertices with outdegree $m$ and assigned the activity $a_{2m+1} = 1$, and balls of type $2m+2$ represent frozen vertices with outdegree $m$ and assigned the activity $a_{2m+2} = 0$. The balls of type $2K+1$ represent active vertices with outdegree $m \geq k$, and given the activity $a_{2K+1} = 1$, while frozen vertices with outdegree $m \geq K$ are represented by balls of type $2K+2$ and have activity $a_{2K+2} = 0$. 

For $m= 0, \ldots, K-1$, if an active vertex $v_n$ with degree $m$ is selected in the tree at step $n-1$, then it either becomes frozen with probability $1-p$ (if $x_n = -1$), or with probability $p$ a new active vertex with outdegree 0 is attached and the outdegree of the vertex is increased by one (if $x_n = 1$). As a consequence, we see that 
\begin{equation}\label{eq:freezingXi}
\text{for } m=0,\ldots,K-1, \qquad \E\xi_{2m+1,j} = \begin{cases}
p & m\neq 0, j=0, \\
-1 & m\neq 0, j=2m+1, \\
p-1 & m=0, j=0, \\
1-p & j = 2m+2, \\
p & j=2m+3.
\end{cases}
\end{equation}
As for active vertices with outdegree $m \geq K$, with probability $1-p$ the vertex becomes frozen (if $x_n = -1$), and with probability $p$ the vertex remains active and a new active vertex with outdegree 0 is added to the tree (if $x_n = +1$). Therefore, 
\begin{equation}\label{eq:freezingXiSpecial}
\E \xi_{2K+1} = \left( \begin{array}{c}
p \\ 0 \\ \vdots \\ 0 \\ -(1-p) \\ 1-p
\end{array}\right).
\end{equation}

\bigskip

\noindent{\bf Normal limit law and convergence of moments.} With the P\'olya urn setup described above, the vector $X_n = (X_{n,1}, \ldots, X_{n,2K+1})^T$ is defined so that $X_{n,2m+1}$ is the number of active vertices in $\mathcal{T}_n(\bm{x})$ with outdegree $m = 0, \ldots, K-1$, $X_{2m+2}$ is the number of frozen vertices with outdegree $m = 0, \ldots, K-1$, and $X_{2K+1}$ is the number of active vertices with outdegree $m \geq K$. Note that as defined, we have that the total activity of the urn is given by 
\[ S_n := \sum_{j=1}^{2K+2} a_j \cdot X_{n,j} = \sum_{j=0}^K X_{n,2j+1} = \max\left(0, 1 + \sum_{i=1}^n x_i\right).\]
We therefore see that the event of non-extinction up to time $n$ is given by 
\[ \mathcal{E}^c_n := \left\{ 1 + \sum_{i=1}^k x_i > 0 \text{ for all } 0 \leq k \leq n\right\}.\]
We can now prove the following result:

\begin{proposition}
Fix $K \geq 1$. Let $\mathcal{T}_0(\bm{x}), \mathcal{T}_1(\bm{x}), \ldots, \mathcal{T}_n(\bm{x}), \ldots,$ be a sequence of uniform attachment trees with freezing, where $\bm{x} = (x_i)_{i=1}^\infty$ is given by i.i.d. distributed $x_i$ with $\P(x_i = +1) =p$ and $\P(x_i = -1) = 1-p$, and let $X_n$ be defined as above. For all $p \in (1/2, 1]$, conditioned on non-extinction, 
\begin{equation}\label{eq:frozCLT}
\frac{X_n - \E[X_n | \mathcal{E}_n^c]}{\sqrt{n}} \xrightarrow{d} \N(0, \Sigma)
\end{equation}
for some covariance matrix $\Sigma$, with convergence of all conditional moments holding as well. Furthermore, conditioned on non-extinction, 
\[ \frac{X_n}{n} \xrightarrow{a.s.} (2p-1)v_1,\]
where $v_1$ is given by
\[ v_1 = \left(
\frac{1}{2} , \frac{1-p}{2(2p-1)} , \frac{1}{4} , \frac{1-p}{4(2p-1)} , \ldots , \frac{1}{2^K} , \frac{1-p}{2^K(2p-1)} , \frac{1}{2^K}, \frac{1-p}{2^K(2p-1)}\right)^T.\]
\end{proposition}

\begin{remark}\label{rmk:degreeFrozen}
Since multivariate normal limit laws hold if and only if normal limit laws hold for all linear combinations of the components, we see immediately that if we let 
\[ Y_n = \left( \begin{array}{c} X_{n,0} + X_{n,1} \\ X_{n,2} + X_{n,3} \\ \vdots \\ X_{n,2K-1} + X_{n,2K} \end{array}\right),\]
 so that $Y_{n,m}$ is the number of vertices (both active and frozen) of outdegree $m+1$, then the convergence 
\begin{equation}\label{eq:rmkY}
\frac{Y_n - \E[Y_n | \mathcal{E}_n^c]}{\sqrt{n}} \xrightarrow{d} \N(0,\Sigma')
\end{equation}
also holds conditioned on non-extinction, along with the convergence of the conditional moments, for the appropriate covariance matrix $\Sigma'$. The joint convergence of infinitely many random variables is by definition the same as joint convergence of any finite subset. Since \eqref{eq:frozCLT} and \eqref{eq:rmkY} hold for any arbitrary $K$, we therefore also have a joint normal limit law for the infinite vector of all degrees in $\mathcal{T}_n(\bm{x})$. 
\end{remark}

\begin{proof}
From \eqref{eq:freezingXi} and \eqref{eq:freezingXiSpecial} along with the activities $a_{2m+1} = 1$ and $a_{2m+2} = 0$, the intensity matrix of the urn is given by 
\[ 
A = \left( \begin{array}{ccccccccc}
p-1 & 0 & p & 0 & \cdots & p & 0 & p & 0 \\
1-p & 0 & 0 & 0 & \cdots & 0 & 0 & 0 & 0\\
p & 0 & -1 & 0 & \cdots & 0 & 0 & 0 & 0\\
0 & 0 & 1-p& 0 & \cdots & 0 & 0 & 0 & 0\\
0 & 0 & p & 0 & \cdots & 0 & 0 & 0 & 0\\
\vdots & \vdots & \vdots & \vdots & \ddots & \vdots & \vdots & \vdots & \vdots\\
0 & 0 & 0 & 0 & \cdots & -1 & 0 & 0 & 0\\
0 & 0 & 0 & 0 & \cdots & 1-p& 0 & 0 & 0\\
0 & 0 & 0 & 0 & \cdots & p & 0 & -(1-p) & 0 \\
0 & 0 & 0 & 0 & \cdots & 0 & 0 & 1-p & 0
\end{array}\right)
\]

To find the eigenvalues of $A$, we use the same procedure as \cite[Lemma 3.2]{DEHO:20}; we perform row and column operations to $A - \lambda I$. For all $i=1,\ldots,2K$, add $a_i$ times row $i$ to row $2K+1$, and afterwards, subtract column $2K+1$ from columns $j=1,\ldots,2K$. The resulting matrix is given by 
\[ A''(\lambda) = \left( \begin{array}{ccccccccc}
-1-\lambda & 0 & 0 & 0 & \cdots & 0 & 0 & p & 0\\
1-p & -\lambda & 0 & 0 & \cdots & 0 & 0 & 0& 0 \\
p & 0 & -1-\lambda & 0 & \cdots & 0 & 0 & 0 & 0\\
0 & 0 & 1-p& -\lambda & \cdots & 0 & 0 & 0 & 0\\
0 & 0 & p & 0 & \cdots & 0 & 0 & 0 & 0\\
\vdots & \vdots & \vdots & \vdots & \ddots & \vdots & \vdots & \vdots & \vdots\\
0 & 0 & 0 & 0 & \cdots & -1 - \lambda & 0 & 0 & 0\\
0 & 0 & 0 & 0 & \cdots & 1-p& -\lambda & 0 & 0\\
0 & 0 & 0 & 0 & \cdots & 0 & 0 & -(1-2p) - \lambda & 0 \\
0 & 0 & 0 & 0 & \cdots & 0 & 0 & 1-p & -\lambda
\end{array}\right).\]
Determinants are unchanged after adding/subtracting a row/column to another. Thus we can read off the characteristic polynomial of $A$ by finding the determinant of $A''(\lambda)$, which by expanding along the last columns and then row $2K-1$, gives 
\[ -\lambda^{K+1}(\lambda + 1)^K(\lambda +1 - 2p).\]
The eigenvalues of $A$ are then given by $\lambda_1 = 2p-1 > 0$ along with $0$ and $-1$.

For the eigenvector $v_1$, we solve the system 
\[ Av_1 = (2p-1)v_1, \qquad \text{and} \qquad a \cdot v_1 = 1.\]
To start we see that 
\begin{equation}\label{eq:v11} -v_{1,1} + p\sum_{m=0}^{K}v_{1,2K+1} = (2p-1)v_{1,1} \Longrightarrow 2p v_{1,1} = p(a \cdot v_1) \Longrightarrow v_{1,1} = \frac{1}{2}.
\end{equation}
Then for $m=1,\ldots,K-1$, 
\begin{equation}\label{eq:v12m+1}
pv_{1,2m-1} - v_{1,2m+1} = (1-2p)v_{1,2m+1} \Longrightarrow v_{1,2m+1} = \frac{v_{1,2m-1}}{2},
\end{equation}
and so from \eqref{eq:v11} and \eqref{eq:v12m+1}, 
\[ \text{for } m=0,\ldots,K-1, \qquad v_{1,2m+1} = \frac{1}{2^{m+1}}.\]
We also see that 
\[ p\, v_{1,2K-1}-(1-p)v_{1,2K+1} = (2p-1)v_{1,2K+1} \Longrightarrow v_{1,2K+1} = v_{1,2K-1} = \frac{1}{2^K}.\]
Finally, for $m=0,\ldots,K$, 
\[ (1-p)v_{1,2m+1} = (1-2p)v_{1,2m+2} \Longrightarrow v_{1,2m+2} = \frac{(1-p)v_{1,2m+1}}{2p-1} = \frac{1-p}{(2p-1)2^{m+1}}.\]

It is easy to verify that our urn satisfies the assumptions (A1) -- (A6) of \cite{JANS:04}, and so by \cite[Theorem 3.21 \& Theorem 3.22]{JANS:04}, conditioned on non-extinction, 
\[ \frac{X_n}{n} \xrightarrow{a.s.} \lambda_1 v_1, \qquad \text{and} \qquad \frac{X_n - n\lambda_1 v_1}{\sqrt{n}} \xrightarrow{d} \N(0,\Sigma).\]
Next we see that $\U{p}$ holds for all $p \geq 2$. Thus by Theorem \ref{thm:exp} and Theorem \ref{prop}, conditioned on non-extinction, 
\[ \frac{X_n - \E[X_n | \mathcal{E}_n^c]}{\sqrt{n}} \xrightarrow{d} \N(0,\Sigma),\]
with the convergence of all conditional moments holding by Theorem \ref{thm:MainThm}.
\end{proof}

\subsection{Hooking networks}\label{sec:hooking}

Let $\chi \geq 0$ and $\rho \in \R$ be two fixed parameters so that $\chi + \rho > 0$, and let $\mathcal{C} = \{G_1, G_2, \ldots, G_m\}$ be a collection of graphs, called {\em blocks}, each with a vertex labelled as the {\em hook} $h_i$ of $G_i$, and each assigned a value $p_i$ so that $p_1 + p_2 + \ldots + p_m = 1$. {\em Hooking networks} are defined as a sequence of random graphs $\mathcal{G}_0, \mathcal{G}_1, \ldots, \mathcal{G}_n, \ldots$ constructed recursively as follows. Starting with $\mathcal{G}_0$ consisting of a single vertex labelled $H$ called the {\em master hook}, $\mathcal{G}_n$ is constructed from $\mathcal{G}_{n-1}$ by first sampling a vertex $v_n$ called a {\em latch} from $\mathcal{G}_{n-1}$, where the probability of selecting $v_n$ is given by 
\[ \frac{\chi \deg(v_n) + \rho}{\sum_{u \in V(\mathcal{G}_{n-1})} \chi \deg(u) + \rho}.\]
Then a block $B_n$ is sampled as a copy of one of the graphs in $\mathcal{C}$, where the probability that $B_n \simeq G_i$ is $p_i$. The hook of $B_n$ is then fused with the latch $v_n$. Recently, results on the degree distributions of hooking networks have been proved, \cite{DEHO:20, MAHM:19}, as have results on the insertion depth; the distance from the block $B_n$ to the master hook $H$ \cite{DESM:24, DEMA:22}. 

An integer $k$ is called an {\em essential} degree if there is a positive probability that at least two vertices have degree $k$ in some $\mathcal{G}_n$; it is shown that only the master hook $H$ can have a non-essential degree \cite[Proposition 3]{DEHO:20}. For a positive integer $r$, let $k_1, \ldots, k_r$ be the $r$ smallest essential degrees that can appear in the sequence of hooking networks, and let $X_n = (X_{n,1}, \ldots, X_{n,r})$, where $X_{n,i}$ is the number of vertices with degree $k_i$ in $\mathcal{G}_n$. 

The degree distribution of hooking networks constructed from a single graph in the collection $\mathcal{C} = \{ G\}$, called {\em self-similar hooking networks}, were studied in \cite{MAHM:19}. The general case was later studied in \cite[Theorem 4]{DEHO:20}, where it was shown that 
\begin{equation}\label{eq:CLTHooking}
\frac{X_n - n\lambda_1 \nu}{\sqrt{n}} \xrightarrow{d} \N(0,\Sigma)
\end{equation}
for $\lambda_1$ and $\nu$ given explicitly in \cite[Section 1.3.1]{DEHO:20}. The `balanced' case was also examined explicitly, where conditions were given to guarantee convergence of moments \cite[Corollary 6]{DEHO:20}. In light of our results in Section \ref{sec:results}, convergence of moments can now be guaranteed for any finite collection of graphs $\mathcal{C}$ to construct hooking networks. 
\begin{proposition}
Let $\mathcal{G}_0, \mathcal{G}_1, \ldots,$ be a sequence of hooking networks constructed as above, and let $X_n = (X_{n,1}, \ldots, X_{n,r})$ be the vector of $r$ smallest essential degrees. The convergence \eqref{eq:CLTHooking} holds with convergence of all moments. Also, $\E X_n = n\lambda_1 \nu + o(\sqrt{n})$, and so $n \lambda_1 \nu$ can be relpaced by $\E X_n$ in \eqref{eq:CLTHooking}.
\end{proposition}

\begin{proof}
It is shown in \cite[Section 3.1.1]{DEHO:20} that $X_n$ can be described as an urn process, and as is shown in \cite[Remark 20]{DEHO:20}, the urn is also balanced in expectation. Since the collection $\mathcal{C}$ is finite, it is also evident that $\U{p}$ holds for all $p \geq 2$. The proposition then follows from \eqref{eq:CLTHooking}, Theorem \ref{thm:exp}, Theorem \ref{prop}, and Theorem \ref{thm:MainThm}.
\end{proof}

\begin{remark}
It is also stated in \cite[Remark 20]{DEHO:20} that the covariance matrix $\Sigma$ can be calculated via the simpler form given by \cite[Lemma 5.4]{JANS:04}; however, as discussed in Remark \ref{rmk:5.4}, this statement does not hold in the general case, but does hold in the `balanced' case satisfying the conditions of \cite[Corollary 6]{DEHO:20}.
\end{remark}

\end{document}